\documentclass[a4paper, 12pt]{amsart}

\usepackage[T1]{fontenc}
\usepackage[linktocpage]{hyperref}
	\hypersetup{
		colorlinks=true,
		linkcolor=blue,
		citecolor=blue,
		filecolor=magenta,
		urlcolor=blue
	}
\usepackage{geometry}       
\usepackage[english]{babel}  
\usepackage{amsmath,amsfonts,amssymb,amsthm}
\usepackage{tikz,tkz-graph,tikz-cd}

\newtheorem{theorem}{Theorem}[section]
\newtheorem{proposition}[theorem]{Proposition}

\newtheorem{corollary}[theorem]{Corollary}
\newtheorem{question}[theorem]{Question}

\theoremstyle{definition}
\newtheorem{definition}[theorem]{Definition}

\theoremstyle{remark}
\newtheorem{remark}[theorem]{Remark}

\newcommand{\CC}{\mathbb{C}}
\newcommand{\ZZ}{\mathbb{Z}}
\newcommand{\PP}{\mathbb{P}}

\newcommand{\A}{\mathcal{A}}
\newcommand{\B}{\mathcal{B}}
\newcommand{\C}{\mathcal{C}}
\newcommand{\D}{\mathcal{D}}
\newcommand{\R}{\mathcal{R}}

\DeclareMathOperator{\tub}{Tub}

\begin{document}

\title{Zariski pairs of conic-line arrangements with a unique conic.}

\author[S. Bannai]{Shinzo Bannai}
	\address[S. Bannai]{
		Department of Applied Mathematics,
		Faculty of Science,
		Okayama University of Science,
		1-1 Ridai-cho, Kita-ku,
		Okayama 700-0005 Japan.
	}
	\email{bannai@ous.ac.jp}

\author[B. Guerville-Ball\'e]{Beno\^it Guerville-Ball\'e}
	\address[B. Guerville-Ball\'e]{}
	\email{benoit.guerville-balle@math.cnrs.fr }

\author[T. Shirane]{Taketo Shirane}
	\address[T. Shirane]{
		Department of Mathematical Sciences,
		Faculty of Science and Technology,
		Tokushima University,
		2-1 Minamijyousanjima-cho,
		Tokushima 770-8506, Japan.
	}
	\email{shirane@tokushima-u.ac.jp}

\thanks{}

\subjclass[2020]{
	14H45, 
	14H30, 
	32Q55, 
	54F65, 
}

\begin{abstract}
	In this note, we present two pairs of conic-line arrangements admitting a unique conic and that form Zariski pairs, both of degree $9$. Their topologies are distinguished using the connected numbers.
\end{abstract}

\maketitle


\section{Introduction}

A \emph{Zariski pair} is a couple of two algebraic plane curves~$\C_1$ and~$\C_2$ --not necessarily irreducible-- such that there exists a homeomorphism form $(\tub(\C_1), \C_1)$ to $(\tub(\C_2),\C_2)$ but no homeomorphism from $(\CC\PP^2,\C_1)$ to $(\CC\PP^2,\C_2)$, where $\tub(\C_i)$ is a tubular neighborhood of~$\C_i$. We often refer to these properties as, having the same \emph{combinatorics} and having different \emph{embedded topology}, respectively. As first noted in~\cite{ArtCogTok}, the first condition can be replaced by the equivalence of combinatorial data given by the local type of each singular point of $\C_i$ and the correspondance between the irreducible components of~$\C_i$ and the local branches of each singular point. The first example of Zariski pair is, of course, due to Zariski~\cite{Zar1,Zar2,Zar3}. It is formed by two irreducible sextics with six cusps, where in the first curve the cusps lie on a conic and not in the second one. The non-existence of a homeomorphism between their embedded topology is obtained by showing that the fundamental groups of their complements $\CC\PP^2\setminus\C_i$ are non-isomorphic --one is Abelian and the other is not.

In~\cite{Deg:isotopy_classification}, Degtyarev proves that the rigid isopoty class of curves of degree at most 5 is determined by the combinatorial data. He also classifies in~\cite{Deg:deformations_sextics} the isotopy classes of the curves of degree~$6$ with simple singularities. In particular, his work implies that the minimal degree for there to exist a Zariski pair is $6$, which is realized by the seminal example of Zariski formed by irreducible curves. On the other end of the spectrum of algebraic plane curves, we find line arrangements whose irreducible components are all of degree~$1$. Currently, the smallest known Zariski pairs of line arrangement have~$11$ lines~\cite{ACCM,Gue:LLN}. By the work of Nazir, Yoshinaga and Fei~\cite{NazYos,Fei}, we know that the smallest possible degree for a Zariski pair of line arrangement is~$10$. The existence of such Zariski pair with $10$~lines is still an open question. In between, one has reduced curves. In~\cite{Art:couples}, Artal gives the first example of a cubic-line Zariski pair formed by one smooth cubic and three inflexional tangents --so they are curves of degree~$6$. In~\cite{Tok:elliptic_surfaces}, Tokunaga constructs two Zariski pairs of conic-lines arrangements of degree~$7$. The first is formed by curves with $3$~conics and $1$~line, while the curves in the second one have $2$~conics and $3$~lines. Since these examples, several other conic-lines Zariski pairs have been discovered, but none of them admits a unique conic. The purpose of this note is to provide two examples of such conic-line Zariski pair with a unique conic. They are both of degree~$9$ with only nodes, ordinary triple points and tacnodes as singularities.

\section{Conic-line Zariski pairs with a unique conic}

Before recalling the definition of the connected number and giving the two examples, we introduce some vocabulary to formulate more precisely the goal of this note.

\begin{definition}
	A \emph{$(d_1,\dots,d_r)$-arrangement} is a reduced algebraic curve with $d_i$ irreducible components of degree $i$, all of them being smooth.
\end{definition}

For example, the Zariski pair given by Artal in~\cite{Art:couples} is formed by $(3,0,1)$-arrangements, while the two examples of Tokunaga in~\cite{Tok:elliptic_surfaces} are composed by $(1,3)$-arrangements and $(3,2)$-arrangements. With this notation a conic-line arrangement with a unique conic is a $(k,1)$-arrangement.

\medskip

We could easily construct a Zariski pair of $(k,1)$-arrangements by taking a Zariski pair of line arrangements and adding a conic intersecting transversally the line arrangements. Nevertheless, in such a trivial example, the conic does not play a role in the difference of topologies. It appears natural to define a minimal Zariski pair as follow.

\begin{definition}
	A Zariski pair $\C_1$, $\C_2$ is \emph{minimal} if no pair of proper sub-curves of $\C_1$ and $\C_2$ is a Zariski pair.
\end{definition}

Using the vocabulary previously introduced, one can reformulate the goal of this paper as follow:
\begin{center}
	\emph{Construct a minimal Zariski pair of $(k,1)$-arrangements, for a small $k$.}
\end{center}

\medskip

\subsection{The connected number}

The embedded topologies of the following examples are distinguished using the \emph{connected number} introduced by the third author in~\cite{Shi:connected_numbers}.
It is an invariant of the embedded topology of reducible algebraic plane curves in smooth varieties. In the following, for simplicity, we adapt the definition and main result of~\cite{Shi:connected_numbers} to the particular case of algebraic plane curves in~$\CC\PP^2$.

\begin{definition}\label{def:connected_numbers}
	Let~$B$ be an algebraic plane curve of $\CC\PP^2$ and let $\Phi:X \rightarrow \CC\PP^2$ be a cyclic cover branched over~$B$. Let~$C$ be an algebraic plane curve of $\CC\PP^2$ such that~$B$ and~$C$ do not have common irreducible components, and $C \setminus B$ is connected. The \emph{connected number} of~$C$ for the cover~$\Phi$ is the number of connected components of $\Phi^{-1}(C \setminus B)$. It is denoted by~$c_\Phi(C)$.
\end{definition}

As explicitly described in the following theorem, the connected number is an invariant of the homeomorphism type of $(\CC\PP^2,B \cup C)$ which respects an order on the irreducible components of the curve~$B \cup C$.

\begin{theorem}\label{thm:invariance}
	Let~$B_1, B_2$ be two curves of $\CC\PP^2$ and let $m \geq 2$ an integer. For each $i\in\{1,2\}$, let $\Phi_i:X_i \rightarrow \CC\PP^2$ be a $\ZZ_m$-cover induced by a surjection $\theta_i:\pi_1(\CC\PP^2\setminus B_i) \twoheadrightarrow \ZZ_m$. Let~$C_1$ be as in Definition~\ref{def:connected_numbers}, and assume that there exists a homeomorphism $h:\CC\PP^2\rightarrow\CC\PP^2$ and an automorphism $\sigma:\ZZ_m \rightarrow \ZZ_m$ such that $h(B_1)=B_2$ and $\sigma \circ \theta_2 \circ h_* = \theta_1$. If we denote $C_2=h(C_1)$ then $c_{\Phi_1}(C_1) = c_{\Phi_2}(C_2)$.
\end{theorem}

To distinguish the two pairs of the present note, we only use surjection~$\theta_i$ on~$\ZZ_2$, so the automorphism~$\sigma$ can only be the identity. Also, the curves~$C$ will be nodal, so we can reformulate~\cite[Corollary~2.6]{Shi:connected_numbers} as follows.

\begin{proposition}\label{propo:computation}
	Let~$B$ and~$C$ be two curves of $\CC\PP^2$ such that~$B$ is of even degree, $C$~is nodal with smooth irreducible components, and~$B$ and~$C$ intersect outside the nodes of~$C$. If~$\Phi_{B}$ is the double cover branched at~$B$ then one has that $c_{\Phi_B}(C)=2$ if and only if there exists a divisor~$D$ on $\CC\PP^2$ with $2 D_{\mid C} = B_{\mid C}$, and $c_{\Phi_B}(C)=1$ otherwise.
\end{proposition}

\begin{remark}\label{rmk:connected_numbers}
	If, for each $P \in B \cap C$, the local intersection multiplicity of~$B$ and~$C$ is~$2$, then the existence of the divisor~$D$ corresponds to the existence of a curve~$\D$ of degree~$\deg (B) / 2$ passing through all the points of~$C \cap B$ which do not contain irreducible components of~$C$.
\end{remark}

\subsection{First example}

Let us consider the algebraic plane curves constructed as follows. First, let~$C$ be the conic defined by the equation $x^2 + y^2 + z^2 - 2xy - 2xz - 2yz = 0$, then consider the four lines $L_1: x = 0$, $L_2: y = 0$, $L_3: x + y - 5z = 0$ and  $L_4: 3x - 2y - 10z = 0$. Last, we define two sets of three lines:\\[5pt]
\indent
\begin{tabular}{cllll}
	$\bullet$ & $L^1_5: x + y + 5z = 0$, \quad 	& $L^1_6: 3x + 3y - 10z = 0$ \quad	& and \quad & $L^1_7: 2x - 3y + 10z = 0$; \\
	$\bullet$ & $L^2_5: 3x - y - 5z = 0$, \quad 	& $L^2_6: 3x + y - 10z = 0$ \quad 	& and \quad & $L^2_7: 6x + y - 10z = 0$.\\[5pt]
\end{tabular}

We denote by~$\C$ the triangle formed by the lines $L_1$, $L_2$ and $L_3$, and by~$\B^1$ (resp.~$\B^2$) the curve defined by~$C$ and~$L_4$ together with $L^1_5$, $L^1_6$ and $L^1_7$ (resp. $L^2_5$, $L^2_6$ and $L^2_7$).

\medskip

The curves~$\C \cup \B^i$, represented in Figure~\ref{fig:ZP1}, share the same combinatorics whose singular points are the ones described below. The intersections between the irreducible components that are not mentioned below are ordinary double points, i.e. nodes.\\[5pt]
\indent
\begin{tabular}{cll}
	$\bullet$ 	&	Triple points:	& $\{L_1,L_4,L^i_5\}$, $\{L_1,L^i_6,L^i_7\}$, $\{L_2,L^i_5,L^i_7\}$, $\{L_2,L_4,L^i_6\}$, $\{L_3,L^i_5,L^i_6\}$,\\
					&							& $\{L_3,L^i_7,C\}$, $\{L_3,L_4,C\}$,\\
	$\bullet$ 	& Tacnodes:			& $\{L_1,C\}$, $\{L_2,C\}$.\\[5pt]
\end{tabular}

\begin{figure}[h!]
	\includegraphics[scale=0.14]{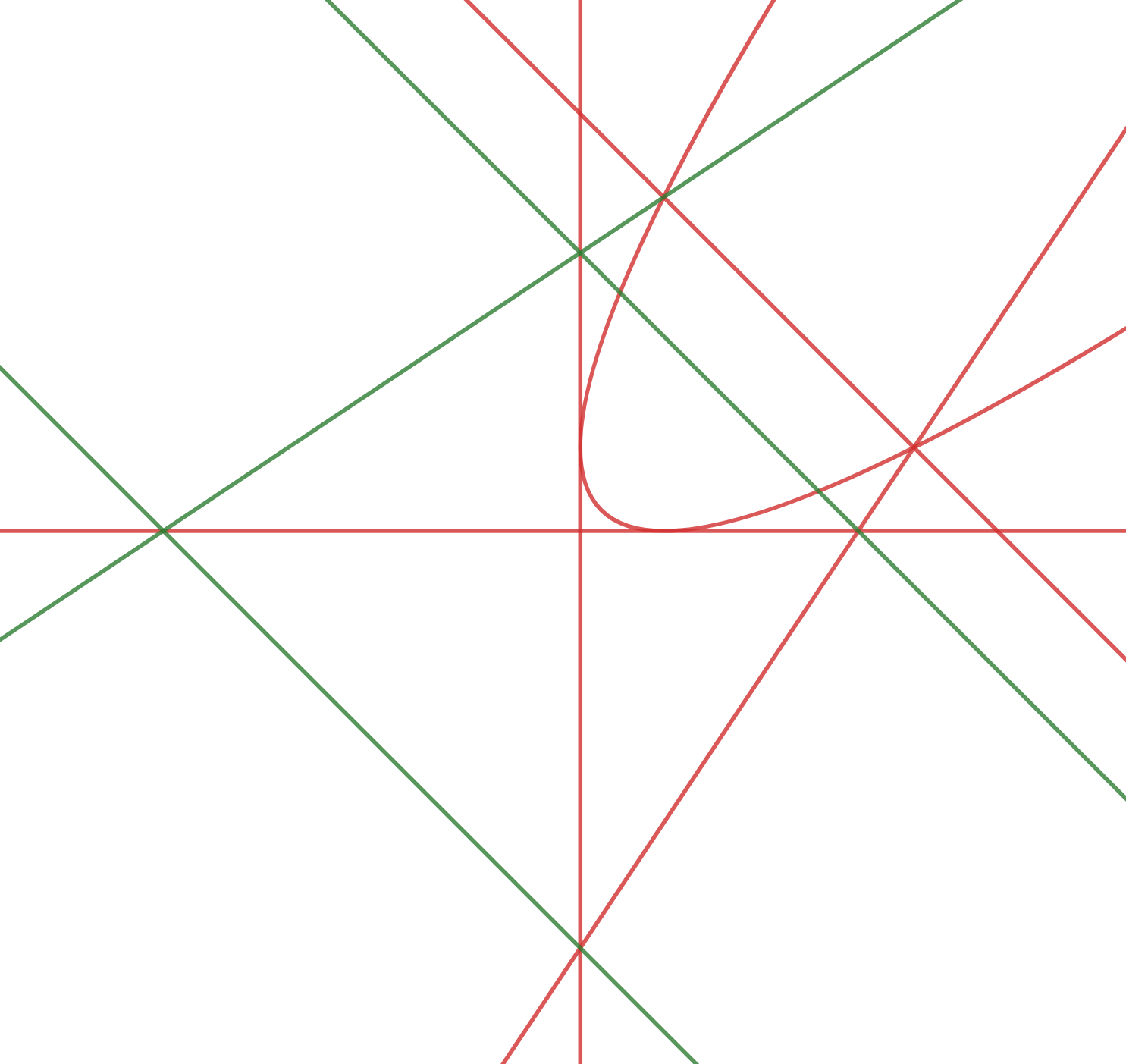}
	\hspace{3cm}
	\includegraphics[scale=0.15]{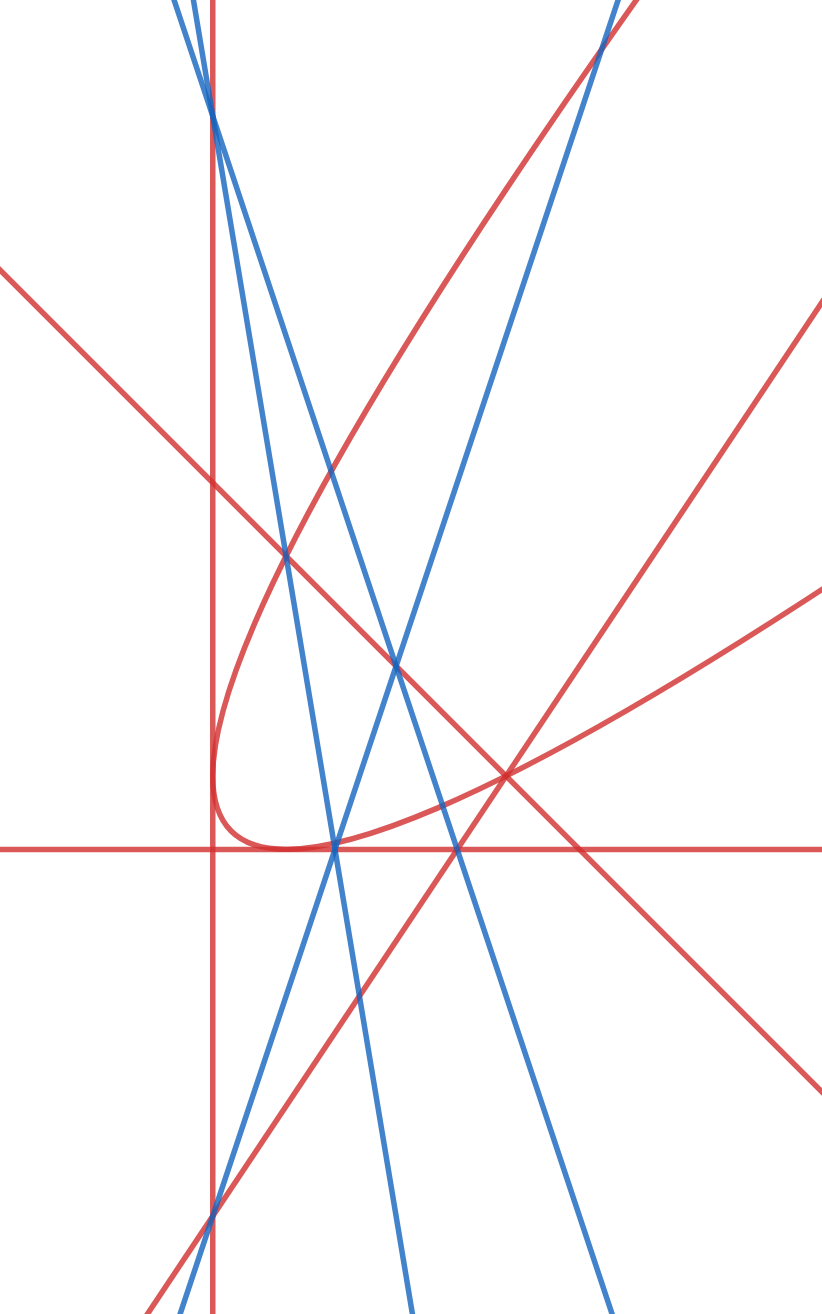}
	\caption{The first Zariski pair. \label{fig:ZP1}}
\end{figure}

\begin{theorem}\label{thm:ZP1}
	There is no homeomorphism $h:\CC\PP^2\rightarrow\CC\PP^2$ such that $h(\C \cup \B^1)=\C \cup \B^2$.
\end{theorem}

\begin{proof}
	Since~$\C$ is the union of three lines in generic position, then it is a nodal curve. Then, the curves~$B_i$ are of even degree and all the intersection points between~$\C$ and~$\B_i$ are outside this nodes. So one can apply Proposition~\ref{propo:computation}. Furthermore, since the intersection points between $B_i$ and $C$ are triple points and tacnodes, then the local intersection multiplicity is~$2$. So, we are under the hypothesis of Remark~\ref{rmk:connected_numbers}. Thus, we have to study the space of degree $\deg (\B^i)/2 = 3$ curves passing through the intersection points of~$\C$ and~$\B^i$, that is, through the $7$~triple points and the $2$~tacnodes listed above.

	We denote them by $P_1,\dots,P_9$, and by $[p_i,q_i,r_i]$ their respective projective coordinates. The dimension of that space is equal to the corank minus~$1$ of the $9 \times 10$ matrix formed by the vectors $(p_i^a q_i^b r_i^c)_{a+b+c=3}$, for $i\in\{1,\dots,9\}$.\footnote{The term~$-1$ is due to the fact that the equation of a curve is well-defined up to scalar.} A direct computation shows that this dimension is~$1$ for $\C \cup \B^1$ and~$0$ for $\C \cup \B^2$. Since $P_1, \dots , P_9$ are on~$\C$ which is of degree~$3$, it implies that for $\C\cup\B^2$, there is no curve of degree~$3$ which pass through these $9$ points and which does not contain $L_i$ for $i\in\{1,2,3\}$, while such a curve exist for $\C\cup\B^1$. By Remark~\ref{rmk:connected_numbers}, we deduce that $c_{\Phi_{\B^1}}(\C) \neq c_{\Phi_{\B^2}}(\C)$.

	Let assume that there exists a homeomorphism $h$ of $\CC\PP^2$ such that $h(\C \cup \B^1) = \C \cup \B^2$. Since $L_1$ and $L_2$ are the only tangent lines of the unique conic $C$ of $\C \cup \B^i$, then $h(L_1 \cup L_2) = L_1 \cup L_2$. Then, one remark that $L_3$ is the only line intersecting $L_1 \cup L_2$ into ordinary double points, so $h(L_3) = L_3$. As a consequence, $h$ should verifies $h(\C) = \C$ and $h(\B^1) = \B^2$ which is in contradiction with Theorem~\ref{thm:invariance}.
\end{proof}

\begin{corollary}
	The curves $\C\cup\B^1$ and $\C\cup\B^2$ form a Zariski pair of $(7,1)$-arrangements.
\end{corollary}

\subsection{Second example}

For this second example, we still consider the conic~$C$ defined by $x^2 + y^2 + z^2 - 2xy - 2xz - 2yz = 0$. Then, we take the four lines $L_1: x + 2y - 4z = 0$, $L_2: x = 0$, $L_3: y = 0$, and $L_4: x - y + 2z = 0$. Last, we have the two sets of three lines:\\[5pt]
\indent
\begin{tabular}{cllll}
	$\bullet$	& $L^1_5: 2x + 2y - 5z = 0$, \quad 	& $L^1_6:x + 7y - 4z = 0$ \quad 	& and \quad 	& $L^1_7: 2x - 2y + z = 0$,\\
	$\bullet$	& $L^2_5: x + 3y - 7z = 0$, \quad 	& $L^2_6: x - 4z = 0$ \quad 	& and \quad 	& $L^2_7: x - y + 5z = 0$.\\[5pt]
\end{tabular}

We set~$\C$ the curve formed by~$C$ and~$L_1$, and we denoted by~$\B^1$ (resp.~$\B^2$) the curve formed by $L_2$, $L_3$, $L_4$ together with $L^1_5$, $L^1_6$ and $L^1_7$ (resp. $L^2_5$, $L^2_6$ and $L^2_7$).

\medskip

The common combinatorics of the curves $\C \cup \B^i$, pictured in Figure~\ref{fig:ZP2}, is given by the singular points below. As for the first example, the missing intersections between irreducible components are ordinary double points.\\[5pt]
\indent
\begin{tabular}{cll}
	$\bullet$	& Triple points:	& $\{L_1,L_2,L_4\}$, $\{L_1,L_3,L^i_6\}$, $\{L_1,L^i_5,L^i_7\}$, $\{L_4,L^i_7,C\}$, $\{L_4,L^i_5,C\}$, \\
					&							& $\{L^i_6,L^i_7,C\}$, $\{L^i_5,L^i_6,C\}$,\\
	$\bullet$	& Tacnodes:  		& $\{L_2,C\}$, $\{L_3,C\}$.\\[5pt]
\end{tabular}

\begin{figure}[h!]
	\includegraphics[scale=0.11]{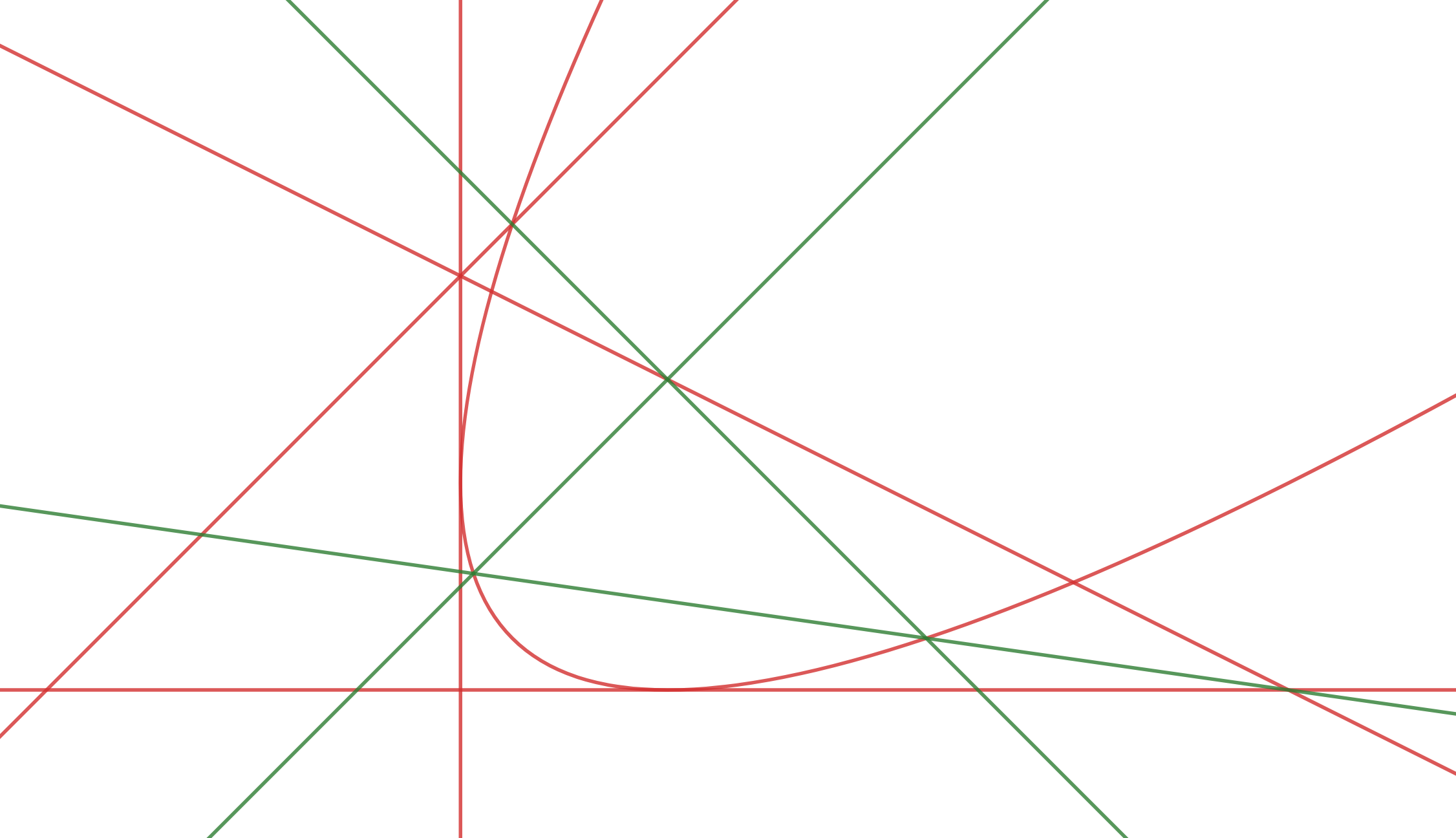}
	\hspace{1cm}
	\includegraphics[scale=0.11]{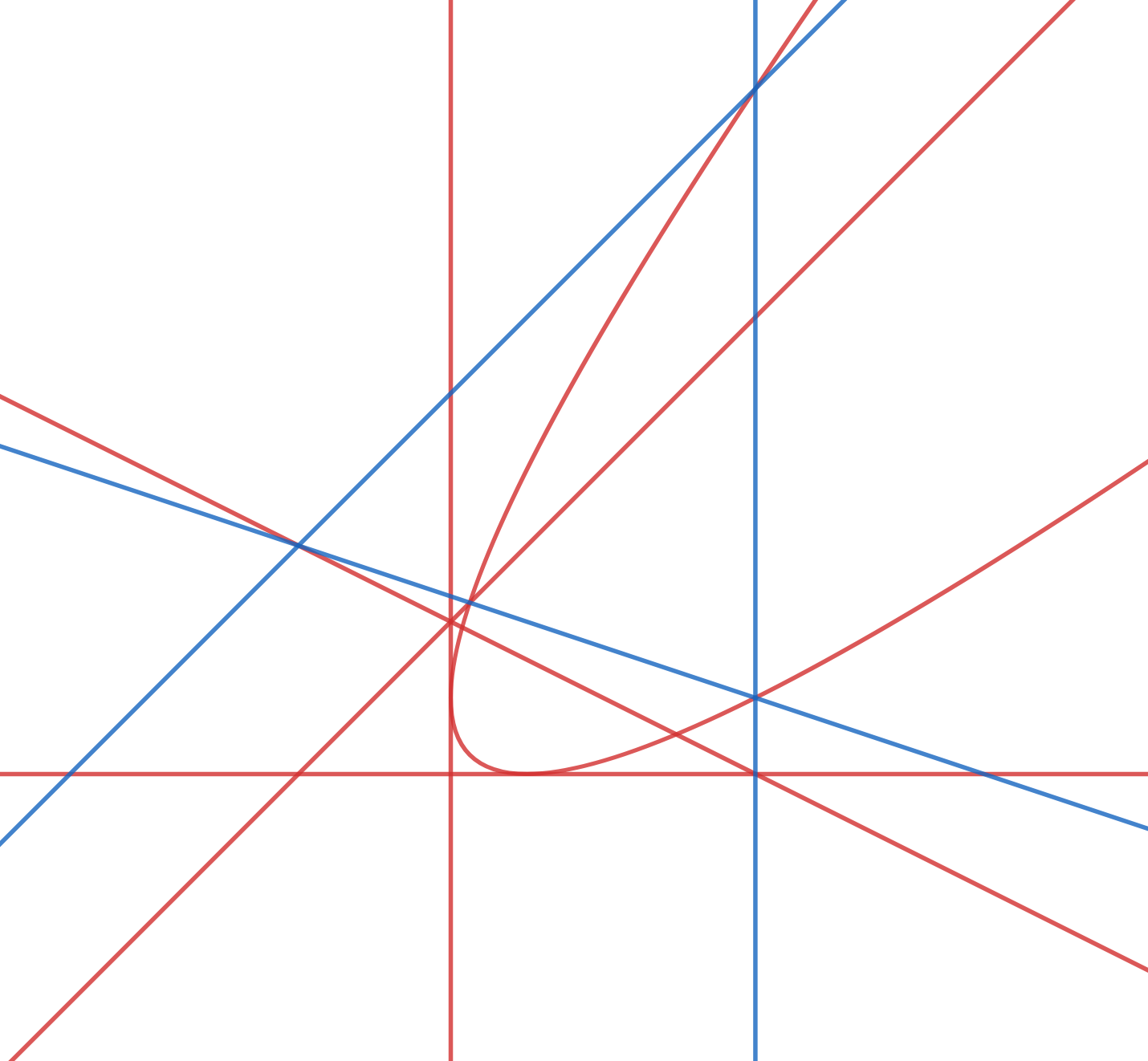}
	\caption{The second Zariski pair. \label{fig:ZP2}}
\end{figure}

\begin{theorem}\label{thm:ZP2}
	There is no homeomorphism $h:\CC\PP^2\rightarrow\CC\PP^2$ such that $h(\C \cup \B^1)=\C \cup \B^2$.
\end{theorem}

\begin{proof}
	We apply the same strategy as for Theorem~\ref{thm:ZP1}. The curve~$\C$ is formed by a smooth conic and a transversal line, so it is a nodal curve. As the intersection points with~$\B^i$ are outside the nodes and the local intersection multiplicities are~$2$, one can apply Remark~\ref{rmk:connected_numbers}. The dimension of the space of curves passing through this $9$~intersection points is~$0$ for $\C \cup \B^1$ and~$1$ for $\C \cup \B^2$. So, $c_{\Phi_{\B^1}}(\C) \neq c_{\Phi_{\B^2}}(\C)$.

	Assume that $h$ is a homeomorphism of $\CC\PP^2$ such that $h(\C \cup \B^1) = \C \cup \B^2$. Since $C$ is the unique conic, and $L_3$ is the only line which intersects $C$ in ordinary double points, then one has $h(\C) = \C$ and $h(\B^1) = \B^2$. This is in contradiction with Theorem~\ref{thm:invariance}.
\end{proof}

\begin{corollary}
	The curves $\C\cup\B^1$ and $\C\cup\B^2$ form a Zariski pair of $(7,1)$-arrangements.
\end{corollary}

\subsection{Minimality}

Let us denote by $\R(\A)$ the realization space of the combinatorics of $\A$. In other words, $\R(\A)$ is the set of all algebraic curves combinatorially equivalent to $\A$. One can adapt~\cite[Proposition~3.3]{NazYos} to the case of $(k,1)$-arrangements as follow --the proof is straightforward and left to the reader.

\begin{proposition}\label{propo:connected_moduli}
	Let $\A$ be a $(k,1)$-arrangement with lines labelled $L_1,\dots,L_k$. We assume that there exists $t_0\in\{1,\dots,k-1\}$ such that $L_t$ intersects transversally the conic of $\A$, for all $t\geq t_0$. For $t\in\{t_0,\dots,k\}$, we define $n_t$ has the number of intersection points between $L_t$ and $\A_t = \A \setminus \{ L_t, \dots, L_k \}$ with a local intersection multipliticty greater or equal than~$2$. If $\R(\A_{t_0})$ is irreducible and, for all $t\in\{t_0,\dots,k\}$, we have $n_t \leq 2$, then $\R(\A)$ is irreducible too and so it is connected.
\end{proposition}

\begin{remark}\label{rmk:connected_moduli}
	If $\A$ is the $(k,1)$-arrangement formed by a conic and one or two tangent lines, then the realization space $\R(\A)$ is irreducible.
\end{remark}

Using this proposition above, one can prove the following theorem.

\begin{theorem}
	The two Zariski pairs described above are minimal.
\end{theorem}

\begin{proof}
	Let denote by $\A_1$ and $\A_2$ one of the examples of Zariski pair above. The arrangements $\A_1\setminus C$ and $\A_2\setminus C$ are line arrangements with $7$~lines. Due to the classification obtained by Nazir and Yoshinaga~\cite{NazYos}, we know that there is no Zariski pair of $7$~lines. By a case-by-case study, we have that, for all $i\in\{1,\dots,7\}$, the $(6,1)$-arrangements $\A_1\setminus L_i$ and $\A_2\setminus L_i$ verify the condition of Proposition~\ref{propo:connected_moduli} --up to a relabelling of the lines. Using Remark~\ref{rmk:connected_moduli}, we deduce that the realization space $\R(\A_1\setminus L_i)=\R(\A_2\setminus L_i)$ is connected. So, $\A_1\setminus L_i$ and $\A_2\setminus L_i$ do not form a Zariski pair. By an iteration of these arguments, we obtain that no proper subarrangements of $\A_1$ and $\A_2$ form a Zariski pair.
\end{proof}

\section{Discussion}

\subsection{Difference between the two examples}

The combinatorics of these pairs can be distinguish from each other as follow. In each one, the conic is unique so an equivalence in the combinatorics should preserve the number of singular points of each type on the conic. In the first example, the conic contains two triples points while it has four in the second one. Equivalently, one could compare the number of double points on the conic.

\subsection{Other methods of distinction}

In the two examples presented, we distinguish the embedded topology using the connected number. It is an invariant of the embedded topology and so it does not provide any information one the complements and their fundemantal groups. Moreover, due to~\cite[Remark~2.2~(iii)]{Shi:connected_numbers}, we know that the connected number is not determined by the fundamental group of the complement.
\begin{question}
	Let~$\A_1$ and~$\A_2$ be one of the previous Zariski pairs, are the fundamental groups $\pi_1(\CC\PP^2\setminus \A_i)$, for $i\in\{1,2\}$, isomorphic?
\end{question}

\begin{remark}
	By~\cite{GueShi}, the Zariski pairs presented here are also distinguished by the linking invariant developped by the second author and Meilhan in~\cite{GueMei}.
\end{remark}

\subsection{Zariski pair of $(k_1,k_2)$-arrangements of minimal degree for fixed $k_2$}

In~\cite{Oka}, Oka presented numerous Zariski pairs of degree~$6$, in particular, there is an example of a Zariski pair of $(0,3)$-arrangements, proving that the minimal degree for such pair is~$6$. For $(k,2)$-arrangement, the example of Tokunaga~\cite{Tok:elliptic_surfaces} implies that the minimal degree of such pairs is~$6$ or $7$. By the work of Nazir and Yoshinaga~\cite{NazYos} and Fei~\cite{Fei}, we know that the minimal degree for a line arrangement is greater than or equal to~$10$, and by Artal, Carmona, Cogolludo and Marco, we have that it is a most~$11$ --see also~\cite{Gue:LLN}. The examples provided in this note prove that the minimal degree for a Zariski pair of $(k,1)$-arrangements is lower than or equal to $9$.
\begin{question}
	Are there Zariski pairs of $(k,1)$-arrangements for $k\in\{5,6\}$?
\end{question}
·


\bibliographystyle{plain}
\bibliography{biblio}

\end{document}